\documentclass[english, leqno]{article}
\usepackage[latin1]{inputenc}
\usepackage[T1]{fontenc}
\usepackage{babel,textcomp}
\usepackage{amssymb}
\usepackage{amsmath}
\usepackage{amsfonts}
\usepackage{graphicx, url}
\usepackage{makeidx}

\numberwithin{equation}{section} %nummerer likninger etter seksjon
\numberwithin{figure}{section} %nummerer figurer etter seksjon

\newtheorem{theorem}{Theorem}[section]

\newtheorem{proposition}[theorem]{Proposition}

\newenvironment {proof} {{\it Proof.}}{\hspace*{\fill}$\Box$\par\vspace{4mm}}

\newcommand{\mc}{\mathcal}
\newcommand{\mb}{\mathbb}

\title{Pricing of claims in discrete time with partial information}
\author{Kristina Rognlien Dahl}

\makeindex{}

\begin{document}

\author{Kristina Rognlien Dahl \thanks{Department of Mathematics, University of Oslo, Pb. 1053 Blindern, 0316 Oslo, Norway (kristrd@math.uio.no, phone: +4722855863). AMS subject classification: 49N15, 60G99, 91G20.}}

\maketitle

\begin{abstract}
\noindent We consider the pricing problem of a seller with delayed price information. By using Lagrange duality, a dual problem is derived, and it is proved that there is no duality gap. This gives a characterization of the seller's price of a contingent claim. Finally, we analyze the dual problem, and compare the prices offered by two sellers with delayed and full information respectively.
\end{abstract}


\section{Introduction}
\label{sec: introduction}

We consider the pricing problem of a seller of a contingent claim $B$ in a financial market with a finite scenario space $\Omega$ and a finite, discrete time setting. The seller is assumed to have information modeled by a filtration $(\mc{G}_t)_t$ which is generated by a delayed price process, so the seller has delayed price information. This delay of information is a realistic situation for many financial market traders. Actually, traders may pay to get updated prices.

The seller's problem is to find the smallest price of $B$, such that there is no risk of her losing money. We solve this by deriving a dual problem via Lagrange duality, and use the linear programming duality theorem to show that there is no duality gap. A related approach is that of King~\cite{King}, where the fundamental theorem of mathematical finance is proved using linear programming duality. Vanderbei and Pilar~\cite{Vanderbei2} also use linear programming to price American warrants.

A central theorem of this paper is Theorem~\ref{thm: selgerspris}, which describes the seller's price of the contingent claim. This generalizes a pricing result by Delbaen and Schachermayer to a delayed information setting (see \cite{DelbaenSchachPrice}, Theorem 5.7). Contrary to what one might guess, this characterization does not involve martingale measures. We can however get an idea of the seller's price by comparing it to that of an unconstrained seller, which is done in Section~\ref{sec: compare}. As one would expect, the seller with delayed information will offer $B$ at a higher price than a seller with full information.

Since the seller's pricing problem is parallel to the buyer's problem, of how much she is willing to pay for the claim, the results will carry through analogously for buyers. This implies that a buyer with delayed information is willing to pay less for the claim than a buyer with full information. Hence, the probability of a seller and buyer with delayed information agreeing on a price is smaller than that of fully informed agents.

This paper considers the case of finite $\Omega$ and discrete time. Although this is not the most general situation, it is of practical use, since one often envisions only a few possible world scenarios, and has a finite set of times where one wants to trade. Also, for this and similar problems in mathematical finance, discretization is necessary to find efficient computational methods.

There are many advantages to working with finite $\Omega$ and discrete time. The information structure of an agent can be illustrated in a scenario tree, making the information development easy to visualize. Conditions on adaptedness and predictability, are greatly simplified. Adaptedness of a process to a filtration means that the process takes one value in each vertex (node) of the scenario tree representing the filtration. Moreover, the general linear programming theory (see Vanderbei~\cite{Vanderbei}) and Lagrange duality framework (see Bertsekas et al.~\cite{Bertsekas}) apply. This allows application of powerful theorems such as the linear programming duality theorem. Also, computational algorithms from linear programming, such as the simplex algorithm and interior point methods, can be used to solve the seller's problem in specific situations. Note that the simplex algorithm is not theoretically efficient, but works very well in practice. Interior point methods, however, are both theoretically and practically efficient. Both algorithms will work well in practical situations where one considers a reasonable amount of possible world scenarios. Theoretically, they may nevertheless be inadequate for a very large number of possible scenarios.

Those familiar with linear programming may wonder why Lagrange duality is used to derive the dual problem instead of standard linear programming techniques. There are two important reasons for this. First of all, the Lagrange duality approach provides better economic understanding of the dual problem and allows for economic interpretations. Secondly, the Lagrange duality method can be explained briefly, and Lagrange methods are familiar to most mathematicians. Hence, using Lagrange duality makes this paper self-contained. The reader does not have to be familiar with linear programming or other kinds of optimization theory.

Other papers discussing the connection between mathematical finance and duality methods in optimization are Pennanen~\cite{Pennanen}, King~\cite{King}, King and Korf~\cite{KingKorf} and Pliska~\cite{Pliska}. Pennanen~\cite{Pennanen} considers the connection between mathematical finance and the conjugate duality framework of Rockafellar~\cite{Rockafellar}. King~\cite{King} proves the fundamental theorem of mathematical finance via linear programming duality, and King and Korf~\cite{KingKorf} derive a dual problem to the seller's pricing problem via the conjugate duality theory of Rockafellar. Pliska~\cite{Pliska} also uses linear programming duality to prove that there exists a linear pricing measure if and only if there are no dominant trading strategies.

Examples of papers considering models with different levels of information in mathematical finance are Di Nunno et. al~\cite{Bernt2}, Hu and {\O}ksendal~\cite{Bernt}, Biagini and {\O}ksendal~\cite{Biagini}, Lakner~\cite{Lakner} and Platen and Rungaldier~\cite{Platen}.

The remaining part of the paper is organized as follows. Section \ref{sec: simplersituation} explains the setting. The financial market is defined, the use of scenario trees to model filtrations is explained and the notation is introduced. Section \ref{sec: Lagrange} covers some background theory, namely Lagrange duality. Section \ref{sec: partialinfo} analyzes the seller's pricing problem with partial information via Lagrange duality. This leads to the central Theorem \ref{thm: selgerspris}. In Section~\ref{sec: compare} we analyze the dual problem, and compare the result of Theorem~\ref{thm: selgerspris} with the price offered by a seller will full information. This leads to Proposition~\ref{thm: compareprice}. The final section, Section \ref{sec: finalremarks}, concludes and poses questions for further research.

\section{The model}
\label{sec: simplersituation}
The financial market is modeled as follows. We are given a probability space $(\Omega, \mc{F}, P)$ consisting of a finite scenario space, $\Omega = \{\omega_1, \omega_2, \ldots, \omega_M\}$, a \\($\sigma$-)algebra (here, there is no difference between $\sigma$-algebras and algebras since $\Omega$ is finite) $\mc{F}$ on $\Omega$ and a probability measure $P$ on the measurable space $(\Omega, \mc{F})$.

 The financial market consists of $N+1$ assets: $N$ risky assets (stocks) and one non-risky asset (a bond). The assets each have a price process $S_{n}(t, \omega)$, $n = 0, 1, \ldots, N$, for $\omega \in \Omega$ and $t \in \{0, 1, \ldots, T\}$ where $T < \infty$, and $S_0$ denotes the price process of the bond. The price processes $S_{n}$, $n = 0, 1, \ldots, N$, are stochastic processes. We denote by $S(t,\omega) := (S_{0}(t,\omega), S_1(t,\omega), \ldots, S_N(t,\omega))$ the vector in $\mb{R}^{N+1}$ consisting of the price processes of all the assets. For notational convenience, we sometimes suppress the randomness, and write $S(t)$ instead of $S(t,\omega)$. Let $(\mc{F}_{t})_{t=0}^{T}$ be the filtration generated by the price processes. We assume that $\mc{F}_0 = \{\emptyset, \Omega\}$ (so the prices at time $0$, $S(0)$, are deterministic) and $\mc{F}_T$ is the algebra corresponding to the finest partition of $\Omega$, $\{\{\omega_1\}, \{\omega_2\}, \ldots, \{\omega_M\}\}$.

 We also assume that $S_0(t,\omega) = 1$ for all $t \in \{0, 1, \ldots,T\}$, $\omega \in \Omega$. This corresponds to having divided through all the other prices by $S_{0}$, and hence turning the bank into the numeraire of the market. This altered market is a discounted market. To simplify notation, the price processes in the discounted market are denoted by $S$ as well. Note that the stochastic process $(S_{n}(t))_{t=0}^{T}$ is adapted to the filtration $(\mc{F}_t)_{t=0}^{T}$.

Consider a contingent claim $B$, i.e., a non-negative, $\mc{F}_T$-measurable random variable. $B$ is a financial asset which may be traded in the market. Therefore, consider a seller of the claim $B$. This seller has price information which is delayed by one time step. We let $(\mc{G}_t)_t$ be the filtration modelling the information structure of the seller. Hence, we let $\mc{G}_0 = \{\emptyset, \Omega\}$, $\mc{G}_t = \mc{F}_{t-1}$ for $t=1, \ldots, T-1$ and $\mc{G}_T = \mc{F}_T$. These assumptions imply that at time $0$ the seller knows nothing, while at time $T$ the true world scenario is revealed. Note that since $\Omega$ is finite, there is a bijection between partitions and algebras (the algebra consists of every union of elements in the partition). The sets in the partition are called blocks. 

One can construct a scenario-tree illustrating the situation, with the tree branching according to the information partitions. Each vertex of the tree corresponds to a block in one of the partitions. Each $\omega \in \Omega$ represents a specific development in time, ending up in the particular world scenario at the final time $T$. Denote the set of vertices at time $t$ by $\mc{N}_{t}$, and let the vertices themselves be indexed by $v = v_1, v_2, \ldots, v_V$.

%ILLUSTRASJON AV SCENARIOTRE:
%
 \begin{figure}[h]
   \setlength{\unitlength}{0.7mm}
 \begin{picture}(50,80)(-40,0)
  \put(20,45){\circle*{3}} % t=1
  \put(40,65){\circle*{3}} % t=2
  \put(40,25){\circle*{3}} % t=2
  \put(60,13){\circle*{3}} % t=3
  \put(60,25){\circle*{3}} % t=3
  \put(60,37){\circle*{3}} % t=3
  \put(60,57){\circle*{3}} % t=3
  \put(60,73){\circle*{3}} % t=3
  \put(20,45){\line(1,1){20}}
  \put(20,45){\line(1,-1){20}}
  \put(40,65){\line(5,2){20}}
  \put(40,65){\line(5,-2){20}}
  \put(40,25){\line(5,3){20}}
  \put(40,25){\line(5,0){20}}
  \put(40,25){\line(5,-3){20}}

%Legger til tekst p� nodene
  \put(2,53){\makebox(0,0){$\Omega = \{\omega_1, \omega_2, \ldots, \omega_5\}$}}
  \put(40,75){\makebox(0,0){$\{\omega_1, \omega_2\}$}}
  \put(40,12){\makebox(0,0){$\{\omega_3, \omega_4, \omega_5\}$}}
  \put(68,73){\makebox(0,0){$\omega_1$}}
  \put(68,57){\makebox(0,0){$\omega_2$}}
  \put(68,37){\makebox(0,0){$\omega_3$}}
  \put(68,25){\makebox(0,0){$\omega_4$}}
  \put(68,13){\makebox(0,0){$\omega_5$}}

%Legger til tidslinje i bunn
  \put(1,4){\line(35,0){80}}
  \put(20,4){\circle*{1}}
  \put(40,4){\circle*{1}}
  \put(60,4){\circle*{1}}

  \put(20,1){\makebox(0,0){$t = 0$}}
  \put(40,1){\makebox(0,0){$t = 1$}}
  \put(60,1){\makebox(0,0){$t = T = 2$}}
 \end{picture}
 \caption{A scenario tree.}
 \label{fig: scenariotre}
 \end{figure}
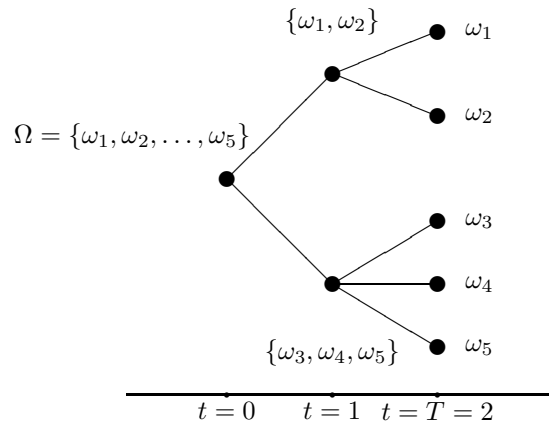

In the example illustrated in Figure \ref{fig: scenariotre}, $V = 8$ and $M = 5$. The filtration $(\mc{G}_t)_{t = 0, 1, 2}$ corresponds to the partitions $\mc{P}_1 = \{\Omega\}$, $\mc{P}_2 = \{\{\omega_1, \omega_2\}, \{\omega_3, \omega_4, \omega_5\}\}$, $\mc{P}_2 = \{\{\omega_1\}, \{\omega_2\}, \ldots, \{\omega_5\}\}$.

\smallskip

Some more notation is useful. The parent $a(v)$ of a vertex $v$ is the unique vertex $a(v)$ preceding $v$ in the scenario tree. Note that if $v \in \mc{N}_t$, then $a(v) \in \mc{N}_{t-1}$. Every vertex, except the first one, has a parent. Each vertex $v$, except the terminal vertices $\mc{N}_{T}$, have children vertices $\mc{C}(v)$. This is the set of vertices immediately succeeding the vertex $v$ in the scenario tree. For each non-terminal vertex $v$, the probability of ending up in vertex $v$ is called $p_{v}$, and $p_{v} = \sum_{u \in \mc{C}(v)} p_{u}$. Hence, from the original probability measure $P$, which gives probabilities to each of the terminal vertices, one can work backwards, computing probabilities for all the vertices in the scenario tree.

\smallskip

 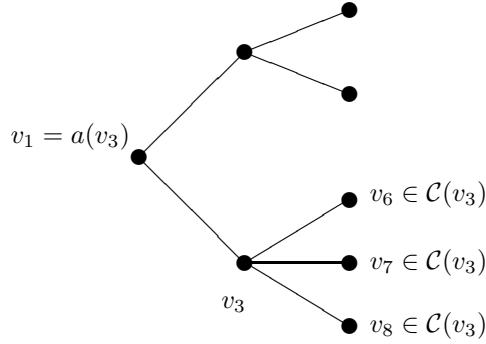
\begin{figure}[h]
   \setlength{\unitlength}{0.7mm}
 \begin{picture}(50,80)(-40,0)
  \put(20,40){\circle*{3}} % t=1
  \put(40,60){\circle*{3}} % t=2
  \put(40,20){\circle*{3}} % t=2
  \put(60,8){\circle*{3}} % t=3
  \put(60,20){\circle*{3}} % t=3
  \put(60,32){\circle*{3}} % t=3
  \put(60,52){\circle*{3}} % t=3
  \put(60,68){\circle*{3}} % t=3
  \put(20,40){\line(1,1){20}}
  \put(20,40){\line(1,-1){20}}
  \put(40,60){\line(5,2){20}}
  \put(40,60){\line(5,-2){20}}
  \put(40,20){\line(5,3){20}}
  \put(40,20){\line(5,0){20}}
  \put(40,20){\line(5,-3){20}}

%Legger til tekst p� nodene
  \put(7,44){\makebox(0,0){$v_1 = a(v_3)$}}
  \put(38,12){\makebox(0,0){$v_3$}}
  \put(75,33){\makebox(0,0){$v_6 \in \mc{C}(v_3)$}}
  \put(75,20){\makebox(0,0){$v_7 \in \mc{C}(v_3)$}}
  \put(75,8){\makebox(0,0){$v_8 \in \mc{C}(v_3)$}}
\end{picture}
\caption{Parent and children vertices in a scenario tree.}
\label{fig: scenariotre3}
\end{figure}

 The adaptedness of the price process $S$ to the filtration $(\mc{F}_{t})_t$ means that, for each asset $n$, there is one value for the price $S_{n}$ in each vertex of the scenario tree. This value is denoted by $S_{n}^{v}$.

\section{The pricing problem with partial information}
\label{sec: partialinfo}

Consider the model and the seller of Section \ref{sec: simplersituation}, with $T \geq 4$. Following the same approach for a smaller $T$ is not a problem, but requires different notation and must therefore be considered separately. Hence, we consider a seller of a contingent claim $B$ who has price information that is delayed with one time step. Recall that the seller's filtration $(\mc{G}_t)_t$ is such that $\mc{G}_0 = \{\emptyset, \Omega\}$, $\mc{G}_t = \mc{F}_{t-1}$ for $t=1, \ldots, T-1$, $\mc{G}_T=\mc{F}_T$.

The pricing problem of this seller is

\begin{equation}
\begin{array}{llrrlll}
\label{eq: prisproblemselger}
&\mbox{\rm{minimize}} &\kappa \\
&\mbox{subject to} \\
&&S_0 \cdot H_0 &\leq &\kappa, \\[\smallskipamount]
&&B_v &\leq &S_v \cdot H_{a_{\mc{G}}(v)} &\mbox{for all } v \in \mc{N}_T^{\mc{G}}, \\[\smallskipamount]
&&S_{\mc{C}_{\mc{G}}(v)} \cdot H_v &= &S_{\mc{C}_{\mc{G}}(v)} \cdot H_{a_{\mc{G}}(v)} &\mbox{for all } v \in \mc{N}_{t}^{\mc{G}} \\[\smallskipamount]
&&&&&\mbox{and for all } \mc{C}_{\mc{G}}(v) \in \mc{N}_{t+1}^{\mc{G}},\\[\smallskipamount] 
&&&&&t=1, \ldots, T-2, \\[\smallskipamount] 
&&S_{\mc{C}_{\mc{F}}(v)} \cdot H_v &= &S_{\mc{C}_{\mc{F}}(v)} \cdot H_{a_{\mc{G}}(v)} &\mbox{for all } v \in \mc{N}_{T-1}^{\mc{G}} \\[\smallskipamount] 
&&&&&\mbox{and for all } \mc{C}_{\mc{F}}(v) \in \mc{N}_{T-1}^{\mc{F}}
\end{array}
\end{equation}

\noindent where the minimization is done with respect to $\kappa \in \mb{R}$ and $H_v \in \mb{R}$ for $v \in \mc{N}_t^{\mc{G}}$, for $t = 0, 1, \ldots, T-1$. Moreover, $\mc{N}_t^{\mc{G}}$ denotes the set of time $t$ vertices in the scenario tree representing the filtration $\mc{G}$, and similarly for the filtration $\mc{F}$. $B_v$ denotes the value of the claim $B$ in the vertex $v \in \mc{N}_T^{\mc{G}}$ (note that each $v \in \mc{N}_T^{\mc{G}}$ corresponds to an $\omega \in \Omega$). Also, $a_{\mc{G}}(v)$ denotes the parent of vertex $v$ w.r.t. the filtration $\mc{G}$ (see Section \ref{sec: simplersituation}). Similarly, $\mc{C}_{\mc{G}}(v)$ and $\mc{C}_{\mc{F}}(v)$ denote the children vertices of vertex $v$ w.r.t. $\mc{G}$ and $\mc{F}$, respectively. 

Hence, the seller's problem is: Minimize the price $\kappa$ of the claim $B$ such that the seller is able to pay $B$ at time $T$ from investments in a self-financing, $\mc{G}$-adapted portfolio that costs less than or equal to $\kappa$ at time $0$. Note that the feasibility constraints in problem~(\ref{eq: prisproblemselger}) imply that the seller acts in a self-financing matter w.r.t. the \emph{actual} prices in the market. Let $\tilde{p}$ denote the seller's price of the claim $B$, so $\tilde{p}$ is the optimal value of problem~(\ref{eq: prisproblemselger}). Problem (\ref{eq: prisproblemselger}) is a linear programming problem. Hence, there are efficient algorithms, such as the simplex algorithm or interior point methods, for solving problem~(\ref{eq: prisproblemselger}), at least if the scenario tree is not too large.

We determine the dual problem of (\ref{eq: prisproblemselger}), using Lagrange duality techniques. In order to use the Lagrange duality method, rewrite the equality feasibility constraints as two inequality constraints. Let $y_0 \geq 0$, $z_v \geq 0$ for all $v \in \mc{N}_T^{\mc{G}}$, $y_v^1, y^2_v \geq 0$ for all $v \in \mc{N}^{\mc{G}}_t$, for $t = 2, 3 \ldots, T-1$ and $w^1_v, w^2_v \geq 0$ for all $v \in \mc{N}_{\mc{F}}^{T-1}$ be the Lagrange dual variables. Let $z$ denote the vector of all the $z_v$'s and similarly $y^i$ $w^i$ the vector of all the $y_v^i$'s and $w^i_v$'s for $i=1,2$. Then, the Lagrange dual problem is

\[
\begin{array}{lll}
\sup_{y_0, z, y^1, y^2,w^1, w^2 \geq 0} &\inf_{\kappa, H} \{\kappa + y_0(S_0 \cdot H_0 - \kappa) + \sum_{v \in \mc{N}_T^{\mc{G}}} z_v(B_v - S_v \cdot H_{a_{\mc{G}}(v)}) \\[\smallskipamount]
&+ \sum_{t = 1}^{T-2} \sum_{v \in \mc{N}_t^{\mc{G}}} \sum_{u \in \mc{C}_{\mc{G}}(v)} (y^1_u - y^2_u) S_u \Delta H_v \\[\smallskipamount]
&+ \sum_{v \in \mc{N}^{\mc{F}}_{T-2}} \sum_{u \in \mc{C}_{\mc{F}}(v)}(w^1_u - w^2_u) S_u \cdot \Delta H_v\} \\[\smallskipamount]
= \sup_{y_0, z \geq 0, y,w} &\{\inf_{\kappa}\{\kappa(1 - y_0)\} + \inf_{H_0}\{H_0 \cdot (y_0 S_0 -\sum_{u \in \mc{C}_{\mc{G}}(1)} y_u S_u)\} \\[\smallskipamount]
&+ \sum_{t=1}^{T-3} \sum_{v \in \mc{N}_t^{\mc{G}}} \inf_{H_v} \{H_v \cdot \sum_{u \in \mc{C}_{\mc{G}}(v)}(y_u S_u \\[\smallskipamount]
&-\sum_{\mu \in \mc{C}_{\mc{G}}(u)} y_{\mu} S_{\mu})\} + \sum_{v \in \mc{N}_{T-2}^{\mc{G}}} \inf_{H_v} \{H_v \cdot \sum_{u \in \mc{C}_{\mc{G}}(v)}(y_u S_u \\[\smallskipamount]
&- \sum_{\mu \in \mc{C}_{\mc{F}}(u)}w_{\mu} S_{\mu})\} + \sum_{v \in \mc{N}_{T-1}^{\mc{G}}} \inf_{H_v}\{H_v \cdot \\[\smallskipamount]
&(\sum_{u \in \mc{C}_{\mc{F}}(v)} w_u S_u - \sum_{u \in \mc{C}_{\mc{G}}(v)} z_u S_u)\} + \sum_{v \in \mc{N}_{T}^{\mc{G}}} z_v B_v\}
\end{array}
\]

\noindent where $y_v := y_v^1 - y_v^2$ and $w_v := w_v^1 - w_v^2$ are free variables, $\Delta H_v := H_v - H_{a_{\mc{G}}(v)}$ and we have exploited that the Lagrange function is separable. 

Consider each of the minimization problems separately. In order to have a feasible dual solution, all of these minimization problems must have optimal value greater than $-\infty$:

\begin{itemize}
\item{$\inf_{\kappa}\{\kappa(1 - y_0)\} > -\infty$ if and only if $y_0 = 1$. In this case, the infimum is $0$.}

\item{$\inf_{H_0}\{H_0 \cdot (y_0 S_0 - \sum_{u \in \mc{C}_{\mc{G}}(1)} y_u S_u)\} > -\infty$ if and only if $y_0 S_0 = \sum_{u \in \mc{C}_{\mc{G}}(1)} y_u S_u$. In this case, the infimum is $0$.}

\item{Note that
\[
\inf_{H_v}  \{H_v \cdot \sum_{u \in \mc{C}_{\mc{G}}(v)}(y_u S_u - \sum_{\mu \in \mc{C}_{\mc{G}}(u)} y_{\mu} S_{\mu})\}> -\infty 
\]
\noindent if and only if $\sum_{u \in \mc{C}_{\mc{G}}(v)}(y_u S_u - \sum_{\mu \in \mc{C}_{\mc{G}}(u)} y_{\mu} S_{\mu})=0$. Therefore, in order to get a dual solution, this must hold for all $v \in \mc{N}_t^{\mc{G}}$ for $t = 1, 2, \ldots, T-3$. In this case, the infima are $0$.}

\item{Furthermore, $\inf_{H_v} \{H_v \cdot \sum_{u \in \mc{C}_{\mc{G}}(v)}(y_u S_u - \sum_{\mu \in \mc{C}_{\mc{F}}(u)}w_{\mu} S_{\mu})\} > -\infty$ if and only if $\sum_{u \in \mc{C}_{\mc{G}}(v)}(y_u S_u - \sum_{\mu \in \mc{C}_{\mc{F}}(u)}w_{\mu} S_{\mu}) = 0$. Again, in this case, the infimum is $0$.}
\item{Finally, $\inf_{H_v}\{H_v \cdot (\sum_{u \in \mc{C}_{\mc{F}}(v)} w_u S_u - \sum_{u \in \mc{C}_{\mc{G}}(v)} z_u S_u)\} > -\infty$ if and only if $\sum_{u \in \mc{C}_{\mc{F}}(v)} w_u S_u = \sum_{u \in \mc{C}_{\mc{G}}(v)} z_u S_u$. Hence, this must hold for all $v \in \mc{N}_{T-1}^{\mc{G}}$. In this case the infimum is $0$.}
\end{itemize}

Hence, the dual problem is 

\begin{equation}
\label{eq: nr1}
\begin{array}{rrllll}
\multicolumn{2}{l}{\sup_{y_0, z \geq 0, y,w} \hspace{0.5cm} \sum_{v \in \mc{N}_T^{\mc{G}}} z_v B_v} \\[\smallskipamount]
\multicolumn{2}{l}{\mbox{subject to}}\\[\smallskipamount]
&y_0 &=& 1, \\[\smallskipamount]
&y_0 S_0 &=& \sum_{u \in \mc{C}_{\mc{G}}(1)} y_u S_u,  \\[\smallskipamount]
&\sum_{u \in \mc{C}_{\mc{G}}(v)}(y_u S_u - \sum_{\mu \in \mc{C}_{\mc{G}}(u)} y_{\mu} S_{\mu}) &=& 0 \hspace{0.2cm} \mbox{ for all } v \in \mc{N}_t^{\mc{G}}, t = 1, 2, \ldots, T-3, \\[\smallskipamount]
&\sum_{u \in \mc{C}_{\mc{G}}(v)}(y_u S_u  - \sum_{\mu \in \mc{C}_{\mc{F}}(u)} w_{\mu} S_{\mu})&=& 0 \hspace{0.2cm} \mbox{ for all } v \in \mc{N}_{T-2}^{\mc{G}},\\[\smallskipamount]
&\sum_{u \in \mc{C}_{\mc{F}}(v)} w_u S_u &=& \sum_{u \in \mc{C}_{\mc{G}}(v)} z_u S_u \hspace{0.2cm} \mbox{ for all } v \in \mc{N}_{T-1}^{\mc{G}}.
\end{array}
\end{equation}

\noindent Note that the dual feasibility conditions are vector equations. From the linear programming duality theorem, see Vanderbei~\cite{Vanderbei}, there is no duality gap. Hence, the optimal value of problem~(\ref{eq: prisproblemselger}) equals the optimal value of problem~(\ref{eq: nr1}).

By analyzing the dual feasibility conditions, we can remove the variable $w$ and rewrite problem~(\ref{eq: nr1}) so that it is expressed using the filtration $(\mc{F}_t)_t$:

\begin{equation}
\label{eq: nr4}
\begin{array}{rrlll}
\multicolumn{2}{l}{\sup_{y_0, z \geq 0, y} \hspace{0.5cm} \sum_{v \in \mc{N}_T^{\mc{G}}}  z_v B_v} \\[\smallskipamount]
\multicolumn{2}{l}{\mbox{subject to}} \\[\smallskipamount]
&y_0 &=& 1, \\[\smallskipamount]
&y_0 S_0 &=& \sum_{u \in \mc{C}_{\mc{F}}(0)} y_u S_u, \\[\smallskipamount]
&\sum_{u \in \mc{C}_{\mc{F}}(v)}(y_u S_u - \sum_{\mu \in \mc{C}_{\mc{F}(u)}} y_{\mu} S_{\mu}) &=& 0 \hspace{0.3cm} \mbox{ for all } v \in \mc{N}_t^{\mc{F}}, \\[\smallskipamount]
&&&\hspace{0.6cm} t = 0, 1, \ldots, T-4, \\[\smallskipamount]
&\sum_{u \in \mc{C}_{\mc{F}}(v)}(y_u S_u  - \sum_{\mu \in \mc{C}_{\mc{F}}(u)} \sum_{\gamma \in \mc{C}_{\mc{F}}(\mu)} z_{\gamma} S_{\gamma})&=& 0 \hspace{0.3cm} \mbox{ for all } v \in \mc{N}_{T-3}^{\mc{F}}.
\end{array}
\end{equation}

It is difficult to interpret problem~(\ref{eq: nr4}) in its present form. It turns out that we can rewrite this problem slightly so that it is easier to understand. Note that

\begin{equation}
\label{eq: star}
 \sum_{v \in \mc{N}_T^{\mc{F}}} z_v S_v = \sum_{u \in \mc{N}_1^{\mc{F}}} y_u S_u = y_0 S_0 
\end{equation}

\noindent where the first equality follows from using the dual feasibility conditions inductively, and summing over all vertices at each time. Equation~(\ref{eq: star}) is a vector equation. Since the market is normalized, the first component of the price process vector is $1$ at each time $t$. Hence, equation~(\ref{eq: star}) implies that $\sum_{v \in \mc{N}_t^{\mc{F}}} z_v = y_0 = 1$ where the final equality uses the first dual feasibility condition. Recall that $z$ is non-negative from problem~(\ref{eq: nr4}). Hence, $z$ can be identified with a probability measure on the terminal vertices of the scenario tree. Denote this probability measure by $Q$. Then, problem~(\ref{eq: nr4}) can be rewritten

\begin{equation}
\label{eq: nr6}
\begin{array}{lrlllll}
\multicolumn{2}{l}{\sup_{Q, y} \hspace{0.5cm} E_Q[B]} \\[\smallskipamount]
\multicolumn{2}{l}{\mbox{subject to}} \\[\smallskipamount]
(i) &S_0& = \sum_{u \in \mc{C}_{\mc{F}}(0)} y_u S_u, \\[\smallskipamount]
(ii) &\sum_{u \in \mc{C}_{\mc{F}}(v)}(y_u S_u - \sum_{\mu \in \mc{C}_{\mc{F}}(u)} y_{\mu} S_{\mu})& = 0 \hspace{0.1cm} \mbox{ for } v \in \mc{N}_t^{\mc{F}}, t = 0, \ldots, T-4, \\[\smallskipamount]
(iii) &\sum_{u \in \mc{C}_{\mc{F}}(v)}y_u S_u & = \sum_{u \in \mc{C}_{\mc{F}}(v)} \sum_{\mu \in \mc{C}_{\mc{F}}(u)} \sum_{\gamma \in \mc{C}_{\mc{F}}(\mu)} q_{\gamma} S_{\gamma} \\[\smallskipamount] 
&&\hspace{0.75cm} \mbox{for } v \in \mc{N}_{T-3}^{\mc{F}}
\end{array}
\end{equation}

\noindent where $Q$ is a probability measure and $q_{\gamma}$ denotes the $Q$-probability of ending up in vertex $\gamma$ at time $T$.

The dual problem is to maximize the expectation of the contingent claim $B$ over a set of probability measures, and some constraints regarding the price process and a free variable $y$. However, there is no martingale measure interpretation of the dual problem. Let $\tilde{d}$ denote the optimal value of the transformed dual problem~(\ref{eq: nr6}).

The previous derivation gives us the following theorem.

\begin{theorem}
 \label{thm: selgerspris}
Consider a seller of a contingent claim $B$ who has partial information in the sense that her price information is delayed by one time step. Let $(\mc{F}_t)_t$ denote the filtration generated by the price process $S$. Then, $\tilde{p} = \tilde{d}$, i.e. the seller's price of $B$ is equal to the optimal value of problem~(\ref{eq: nr6}).
\end{theorem}

Note that for a specific problem, one can solve problem~(\ref{eq: nr6}) using the simplex algorithm or interior point methods (for a reasonably sized scenario tree). Also, the same kind of argument as above can be done from the buyer's point of view, yielding dual problem similar to problem~(\ref{eq: nr6}), but with infimum instead of the supremum.

\section{Some comments on the dual problem}
\label{sec: compare}

\subsection{Connection to full information}
From Delbaen and Schachermayer~\cite{DelbaenSchachPrice} (or a derivation similar to that of Section~\ref{sec: partialinfo}), we know that the seller's price of $B$ with full information is 
\begin{equation}
\label{eq: fullinfo}
 \alpha := \sup_{Q \in \mc{M}(S, \mc{F})} E_Q[B]
\end{equation}
\noindent where $\mc{M}(S, \mc{F})$ denotes the set of equivalent martingale measures w.r.t. the filtration $(\mc{F}_t)_t$. In the following, assume there exists a $Q \in \mc{M}(S, \mc{F})$. From \cite{DelbaenSchachPrice}, this means that there is no arbitrage in the market.

\begin{theorem}
 \label{thm: compareprice}
The difference between the price of $B$ offered by a seller with delayed information and a seller with full information is
\begin{equation}
\label{eq: compareprice}
\tilde{d} - \alpha \geq 0.
\end{equation}
\end{theorem}

\begin{proof}
 From the definition of $\tilde{d}$ and $\alpha$, it suffices to prove that each $Q \in \mc{M}(S, \mc{F})$ corresponds to a solution $y, \tilde{Q}$ of problem~(\ref{eq: nr6}). Hence, let $Q \in \mc{M}(S, \mc{F})$. Define $\tilde{Q} := Q$, and for each $v \in \mc{N}_{T-1}^{\mc{F}}$, define $y_v := \sum_{u \in {C}_{\mc{F}}(v)} \tilde{q}_u$. Similarly, for each $v \in \mc{N}^{\mc{F}}_{T-2}$, define $y_v := \sum_{u \in {C}_{\mc{F}}(v)} y_u$. Iteratively, we define $y_v := \sum_{u \in {C}_{\mc{F}}(v)} y_u$ for each $v \in \mc{N}_{t}^{\mc{F}}$, $t=0, \ldots, T-3$. We would like to show that $\tilde{Q}, y$ are feasible for problem~(\ref{eq: nr6}). 
\begin{itemize}
 \item[$(i)$]: Since $Q \in \mc{M}(S, \mc{F})$, $S_0 = E_Q[S_1 | \mc{F}_0]$, which from the definition of conditional expectation implies $(i)$.
\item[$(ii)$]: $Q \in \mc{M}(S, \mc{F})$ implies that $E_Q[S_{t+1} | \mc{F}_t] = S_t$. Hence, from the definition of conditional expectation, $y_u S_u = \sum_{\mu \in \mc{C}_{\mc{F}}(u)} y_{\mu} S_{\mu}$ for all $u \in \mc{N}_t^{\mc{F}}$, so $(ii)$ holds.
\item[$(iii)$]: Again, since $Q \in \mc{M}(S, \mc{F})$, $E_Q[S_T | \mc{F}_{T-2}] = S_{T-2}$. Hence, \\
$y_u S_u = \sum_{\mu \in \mc{C}_{\mc{F}}(u)} \sum_{\gamma \in \mc{C}_{\mc{F}}(\mu)} q_{\gamma} S_{\gamma}$, so $(iii)$ holds.
\end{itemize}
Hence, the theorem follows.
\end{proof}
\noindent The difference  in Theorem~\ref{thm: compareprice} can be computed for specific examples.

Theorem~\ref{thm: compareprice} implies that, as one would expect, the seller with only partial information will demand a higher price for $B$ than a fully informed seller. As in Section~\ref{sec: partialinfo}, the same kind of argument goes through for a buyer of the claim. Hence, the probability of a seller and buyer agreeing on a price of the claim is smaller in a market with delayed information, than in the fully informed case.

\subsection{A closer bound}
\label{sec: skranke}
We can find an interpretable problem which has optimal value closer to that of problem~(\ref{eq: nr6}) than the full information problem~(\ref{eq: fullinfo}). Consider the following optimization problem

\begin{equation}
 \label{eq: squeeze}
\begin{array}{lrlll}
\sup_{Q} &E_Q[B] \\[\smallskipamount]
\mbox{subject to} \\[\smallskipamount]
&S_0 &=& E_Q[S_1], \\[\smallskipamount]
&E_Q[S_{t+1} | \mc{F}_t] &=& E_Q[S_{t+2} | \mc{F}_t] &\mbox{for } t=0,1, \ldots, T-4,\\[\smallskipamount]
&E_Q[S_{T-2} | \mc{F}_{T-3}] &=& E_Q[S_T | \mc{F}_{T-3}].
\end{array}
\end{equation}
\noindent Let $\beta$ denote the the optimal value of problem~(\ref{eq: squeeze}).

\begin{theorem}
 \label{prop: squeeze}
The optimal value of problem~(\ref{eq: squeeze}) lies between the price of $B$ offered by the seller with full information and the price offered by the seller with delayed information, i.e.,
\[
\alpha \leq \beta \leq \tilde{d}
\]
\end{theorem}

\begin{proof}
Clearly, $\alpha \leq \beta$, from the definition of $\mc{M}(S, \mc{F})$.

It remains to prove that $\beta \leq \tilde{d}$. Hence, consider $Q$ feasible in problem~(\ref{eq: squeeze}). It suffices to prove that $Q$ corresponds to a feasible solution $\tilde{Q}, y$ for problem~(\ref{eq: nr6}). Define $\tilde{Q}$ and $y$ as in the proof of  Proposition~\ref{thm: compareprice}. We check the feasibility constraints of problem~(\ref{eq: nr6}).
\begin{itemize}
 \item[$(i)$]: Since $Q$ is feasible in (\ref{eq: squeeze}), $S_0 = E_Q[S_1]$. Hence, from the definition of conditional expectation, $S_0 = \sum_{u \in \mc{C}_{\mc{F}}(0)} y_u S_u$.
\item[$(ii)$]: Again, since $Q$ is feasible in (\ref{eq: squeeze}), $E_Q[S_{t+1} | \mc{F}_t] = E_Q[S_{t+2} | \mc{F}_t]$ for $t=0,1, \ldots, T-4$. From the definition of conditional expectation, this implies that $\sum_{u \in \mc{C}_{\mc{F}}(v)}(y_u S_u - \sum_{\mu \in \mc{C}_{\mc{F}}(u)} y_{\mu} S_{\mu}) = 0$ for all $v \in \mc{N}_t^{\mc{F}}$, $t = 0, \ldots, T-4$. Hence, $(ii)$ holds.
\item[$(iii)$]: $(iii)$ follows similarly from the feasibility of $Q$ in (\ref{eq: squeeze}) and the definition of conditional expectation.
\end{itemize}
Hence, $\beta \leq \tilde{d}$.
\end{proof}

\section{Final remarks}
\label{sec: finalremarks}

 A main idea of this paper has been to illustrate the close connection between pricing problems in mathematical finance and duality methods in optimization.

The results of this paper can actually be generalized to a model with arbitrary scenario space $\Omega$ by using the conjugate duality theory of Rockafellar~\cite{Rockafellar}. This is a work in progress.

Some questions open for further research are: 
\begin{itemize}
 \item{Can these results be generalized to a model with continuous time, possibly using a discrete time approximation?}
 \item{Is it possible to characterize the partially informed seller's dual problem more explicitly?}
\end{itemize}

\smallskip

\textbf{Acknowledgements:} We would like to thank an anonymous referee for careful reading and many helpful suggestions. We are also very grateful to Professor Bernt {\O}ksendal (University of Oslo) for several useful comments.

\begin{appendix}
 \section{Lagrange duality}
\label{sec: Lagrange}

This section reviews some basic ideas and results concerning Lagrange duality which will be useful in the following. For more on Lagrange duality and optimization theory, see Bertsekas et al.~\cite{Bertsekas}.

Let $X$ be a general inner product space with inner product $\langle \cdot,\cdot \rangle$. Consider a function $f: X \rightarrow \mb{R}$ and the very general optimization problem

\begin{equation}
\label{eq: opprinneligligning}
\mbox{minimize } f(x) \mbox{ subject to  } g(x) \leq 0,\mbox{ } x \in S
\end{equation}

\noindent where $g$ is a function such that $g: X \rightarrow \mb{R}^n$ and $S$ is a non-empty subset of $X$. Here, $g(x) \leq 0$ means component-wise inequality. This will be called the primal problem.

Define the Lagrange function, $L(x,\lambda)$, for $\lambda \in \mb{R}^n$, $\lambda \geq 0$ (component-wise), to be
 \begin{eqnarray}
 L(x, \lambda) &=& f(x) + \lambda \cdot g(x). \nonumber
  \end{eqnarray}
\noindent where $(\cdot)$ denotes Euclidean inner product.

Then, for all $x \in X$ such that $g(x) \leq 0$ (component-wise) and all $\lambda \in \mb{R}^n$, $\lambda \geq 0$, we have $L(x, \lambda) \leq f(x)$. This motivates the definition $L(\lambda) := \inf_{x \in S} L(x, \lambda)$ for all $\lambda \geq 0$ (note that $L(\lambda) = -\infty$ is possible), and the Lagrange dual problem
\[
\sup_{\lambda \geq 0} L(\lambda).
\]
This gives the following result called weak Lagrange duality.
\begin{proposition}
\label{thm: WeakLagrangeDuality}
\[
\begin{array}{lll}
\sup \{L(\lambda) : \lambda \geq 0\} &\leq& \inf\{f(x) : g(x) \leq 0, \mbox{ } x \in S\}.
\end{array}
\]
\end{proposition}

Hence, the Lagrange dual problem gives the greatest lower bound on the optimal value of problem (\ref{eq: opprinneligligning}), based on $L$. Often the Lagrange dual problem is separable, and therefore fairly easy to solve. For some problems, one can proceed to show duality theorems, proving that $\sup_{\lambda \geq 0} L(\lambda) = \inf\{f(x) : g(x) \leq 0, \mbox{ } x \in S\}$. In this case, one says that there is no duality gap. This typically occurs in convex optimization problems under certain assumptions. For instance, the linear programming duality theorem (see Vanderbei~\cite{Vanderbei}) may be derived using Lagrange duality.

\end{appendix}


\begin{thebibliography}{99}



\bibitem{Bertsekas}
Bertsekas, D. P., A.~Nedic, A.E.~Ozdaglar (2003):
\newblock \emph{Convex Analysis and Optimization},
\newblock Belmont: Athena Scientific, 1st ed.

\bibitem{Biagini}
Biagini, F., B.~{\O}ksendal (2005):
\newblock A general stochastic calculus approach to insider trading,
\newblock \emph{Applied Mathematics and Optimization}, vol. 52,  no. 2, 167--181.


\bibitem{BV}
Boyd, S., L.~Vandenberghe (2004):
\newblock \emph{Convex Optimization},
\newblock Cambridge: Cambridge University Press, 1st ed.


\bibitem{DelbaenSchachPrice}
Delbaen, F., W.~Schachermayer (1994):
\newblock A general version of the fundamental theorem of asset pricing,
\newblock \emph{Mathematische Annalen}, vol. 300,  no. 3, 463--520.


\bibitem{Bernt2}
Di Nunno, G., A.~Kohatsu-Higa, T.~Meyer-Brandis, B.~{\O}ksendal, F.N.~Proske, A.~Sulem (2008):
\newblock Anticipative stochastic control for L\'{e}vy processes with application to insider trading, \emph{Mathematical Modelling and Numerical Methods in Finance}, vol. 15,
\newblock Amsterdam: North Holland, 573--595.


\bibitem{Bernt}
Hu, Y., B.~{\O}ksendal (2008):
\newblock Partial Information Linear Quadratic Control for Jump Diffusions,
\newblock \emph{SIAM Journal on Control and Optimization}, vol.  47,  no. 4, 1744--1761.





\bibitem{King}
King, A.J. (2002):
\newblock Duality and martingales: A stochastic programming perspective on contingent claims,
\newblock  \emph{Mathematical Programming}, vol. 91,  no. 3, Ser. B, 543--562.


\bibitem{KingKorf}
King, A. J., and L. Korf (2001):
\newblock Martingale pricing measures in incomplete markets via stochastic programming duality in the dual of $\mc{L}^{\infty}$,
\newblock preprint, University of Washington.

\bibitem{Lakner}
Lakner, P. (1995):
\newblock Utility maximization with partial information,
\newblock  \emph{Stochastic Processes and their Applications}, vol.  56,  no. 2, 247--273.



\bibitem{Pennanen}
Pennanen, T. (2011):
\newblock Convex duality in stochastic optimization and mathematical finance,
\newblock \emph{Mathematics of Opererations Research}, vol.  36,  no. 2, 340--362.




\bibitem{Platen}
Platen, E., W.J.~Rungaldier (2007):
\newblock A benchmark approach to portfolio optimization under partial information,
\newblock preprint, University of Technology Sydney.


\bibitem{Pliska}
Pliska, S. (1997):
\newblock  \emph{Introduction to Mathematical Finance},
\newblock Malden: Blackwell Publishing, 1st ed.

\bibitem{Rockafellar}
Rockafellar, R. T. (1974):
\newblock {\em Conjugate Duality and Optimization},
\newblock Philadelphia: Society for Industrial and Applied Mathematics, 1st ed.





\bibitem{Vanderbei}
Vanderbei, R. J. (2008):
\newblock {\em Linear Programming: Foundations and Extensions},
\newblock Berlin: Springer, 3rd ed.

\bibitem{Vanderbei2}
Vanderbei, R. J., M. \c{C}- Pinar  (2009):
\newblock Pricing American perpetual warrants by linear programming,
\newblock \emph{SIAM Rev.}, no. 4, 767--782.



\bibitem{Oksendal}
{\O}ksendal, B. (2007):
\newblock {\em Stochastic Differential Equations},
\newblock Berlin: Springer, 6th ed.


\end{thebibliography}
\end{document}